\documentclass[reqno,12pt]{amsart}
\usepackage{amsmath,amsthm,amssymb,amsfonts,amscd}
\usepackage{epsfig}
\usepackage{xypic}
\numberwithin{equation}{section}

\theoremstyle{plain}
\newtheorem{theorem}{Theorem}

\newtheorem{proposition}[theorem]{Proposition}
\newtheorem*{theorem*}{Theorem}
\newtheorem*{conjecture*}{Conjecture}

\theoremstyle{definition}
\newtheorem{remark}[theorem]{Remark}
\newtheorem{example}[theorem]{Example}
\newtheorem*{definition}{Definition}

\newcommand{\CC}{{\mathbb{C}}}
\newcommand{\HH}{{\bf H}}
\newcommand{\PP}{{\mathbb{P}}}
\newcommand{\QQ}{{\mathbb{Q}}}

\newcommand{\RR}{{\mathbb{R}}}
\newcommand{\ZZ}{{\mathbb{Z}}}

\newcommand{\ff}{{\bf f}}
\newcommand{\hh}{{\bf h}}

\begin{document}
\title[Strange duality between complete intersection singularities]{Strange duality between the quadrangle complete intersection singularities}
\author{Wolfgang Ebeling and Atsushi Takahashi}
\address{Institut f\"ur Algebraische Geometrie, Leibniz Universit\"at Hannover, Postfach 6009, D-30060 Hannover, Germany}
\email{ebeling@math.uni-hannover.de}
\address{
Department of Mathematics, Graduate School of Science, Osaka University, 
Toyonaka Osaka, 560-0043, Japan}
\email{takahashi@math.sci.osaka-u.ac.jp}
\date{}
\begin{abstract} 
There is a strange duality between the quadrangle isolated complete intersection singularities discovered by the first author and C.~T.~C.~Wall. We derive this duality from the mirror symmetry, the Berglund-H\"ubsch transposition of invertible polynomials, and our previous results about the strange duality between hypersurface and complete intersection singularities using matrix factorizations of size two.
\end{abstract}
\maketitle
\section*{Introduction}

V.~I.~Arnold \cite{A} observed a strange duality  between the 14 exceptional unimodal singularities. It is well known that this  duality is a special case of the Berglund-H\"ubsch duality of invertible polynomials, see e.g. \cite{ET1}. C.~T.~C.~Wall and the first author \cite{EW} discovered an extension of this duality embracing on one hand series of bimodal hypersurface singularities and on the other hand, isolated complete intersection singularities (ICIS) in $\CC^4$. The duals of the ICIS are not themselves singularities but are virtual ($k=-1$) cases of series (e.g. $J_{3,k}$, $k \geq 0$) of bimodal singularities. In \cite{EW}, the $k=-1$ cases of the series were called {\em virtual singularities} and Milnor lattices were associated to them, but they do not coincide with the Milnor lattices of the germs at the origin by setting $k=-1$ in Arnold's equations of the series, which are exceptional unimodal singularities with a smaller Milnor number.
In \cite{ET3}, we showed that the virtual singularities exist in the sense that the equations have to be considered as global polynomials and we derived this extension from the mirror symmetry and the Berglund-H\"ubsch duality of invertible polynomials.

Arnold's 14 exceptional unimodal singularities are {\em triangle} hypersurface singularities, i.e., they are weighted homogeneous singularities obtained from triangles in the hyperbolic plane. More precisely, they are determined by triangles with angles $\frac{\pi}{b_1},\frac{\pi}{b_2},\frac{\pi}{b_3}$, where $b_1,b_2,b_3$ are positive integers called the {\em Dolgachev numbers} of the singularity.
The $k=0$ elements of the bimodal series are {\em quadrangle} hypersurface singularities, i.e., they are related in a similar way to quadrangles in the hyperbolic plane. They are determined by four positive integers  $b_1,b_2,b_3,b_4$. For 6 quadruples $(b_1,b_2,b_3,b_4)$, the corresponding quadrangle singularities are hypersurface singularities. The dual ICIS are  triangle complete intersection singularities in $\CC^4$. There are 8 of them determined by 8 triples $(b_1,b_2,b_3)$. For another 5 quadruples $(b_1,b_2,b_3,b_4)$, the quadrangle singularities are ICIS. These singularities will be considered in this paper. They are again the $k=0$ elements of certain series of singularities. These series are the 8 series of $\mathcal{K}$-unimodal ICIS in Wall's classification \cite{W}.
Wall and the first author  also observed a duality between the  $k=-1$ cases of these series (see also \cite[Sect.~3.6]{E1}). They were called {\em virtual singularites} as well. The objective of this paper is to show that these singularities exist as well and to derive this duality from the Berglund-H\"ubsch duality, too.   The quadrangle complete intersection singularities together with their Dolgachev numbers are listed in Table~\ref{TableICISk=0}. The virtual singularities obtained by setting $k=-1$ in the equations are listed in Table~\ref{TabICIS}.

We derive this duality from our paper \cite{ET3}. An important tool are matrix factorizations of size two. In \cite{ET4}, we showed that such a matrix factorization can be considered as an inverse to
Wall's reduction of complete intersection singularities to hypersurface singularities \cite{WPLMS}.  
We proceed as follows. In \cite{ET3}, we classified certain $4 \times 3$-matrices which provided the duality to complete intersection singularities. Here we consider the polynomials determined by these matrices for the bimodal series. We determine the matrix factorizations of size two of the corank 3 polynomials. It turns out that we get exactly 8 possibilities which correspond to the 8 series of ICIS. We show that one can associate $4 \times 4$-matrices to these equations
such that the duality is given by the Berglund-H\"ubsch transposition of these matrices. Using the definition of the virtual singularities in \cite{ET3} and matrix factorizations again, we define virtual complete intersection singularities. 

Similarly as in \cite{ET3}, we associate Dolgachev numbers to the virtual singularities. These are two pairs of numbers corresponding to a decomposition of the equations into two parts. We also associate Gabrielov numbers to the virtual singularities by considering deformations to cusp singularities. We consider the second function on the zero set of the first function. Again this function has to be considered as a global function. It turns out that these functions have, besides an isolated critical point at the origin, additional critical points outside the origin. We consider Coxeter-Dynkin diagrams of distinguished bases of thimbles corresponding to these functions taking the additional critical points into account. 

We show that  the Dolgachev numbers of a virtual singularity are the Gabrielov numbers of the dual one, and vice versa. Moreover, we show that an analogue of \cite[Theorem~6]{ET3} holds: the reduced zeta function of the monodromy of a virtual singularity coincides with the product of the Poincar\'e series of the $k=0$ element of the dual series and a polynomial related to its Dolgachev numbers. 

\begin{table}
\begin{tabular}{|c|c|c|c|c|}  \hline
Name   & Equations  & Restrictions & Dol  & Weights \\ 
\hline
$J'_{2,0}$ & $(xy+w^2, ax^5+xw^2+yw+z^2)$ & $a \neq 0,-\frac{4}{27}$ & 2,2,2,6 & $2,6,5,4;8,10$\\ 
$K'_{1,0}$ & $(xy+w^2, ax^4+xw^2+y^2+z^2)$ & $a \neq 0,\frac{1}{4}$ & 2,2,4,4 & $2,4,4,3;6,8$ \\ 
$L_{1,0}$ & $(xy+zw, ax^4+xw^2+yw+z^2)$ & $a \neq 0,-1$  & 2,2,3,5  & $2,5,4,3;7,8$ \\ 
$M_{1,0}$ & $(2xy+w^2-z^2, ax^2z+x^2w+2yw)$ & $a \neq 0, \pm 1$ & 2,3,3,4 & $2,4,3,3;6,7$ \\ 
$I_{1,0}$ & $(x(y-z)+w^3, aw^3+y(z-x))$ & $a \neq 0,1$ & 3,3,3,3  & $3,3,3,2;6,6$ \\ 
\hline
\end{tabular}
\caption{The elements with $k=0$ of the series}
\label{TableICISk=0}
\end{table}

\begin{table}
\begin{center}
\begin{tabular}{|c|c|}
\hline
Name  & Equation   \\
\hline
$J'_{2,-1}$ & $(xy+w^2,x^4+xw^2 + yw +z^2)$    \\
$K'_{1,-1}$ & $(xy+w^2,x^3+xw^2+y^2+z^2)$   \\
$K^\flat_{1,-1}$ & $(xy+w^2,x^4+x^2w+2xw^2+y^2+z^2)$   \\
$L_{1,-1}$ & $(xy+zw,x^3+xw^2+yw+z^2)$   \\
$L^\sharp_{1,-1}$ & $(xy+zw,-x^4+x^2w+xw^2+yw+z^2)$     \\
$M_{1,-1}$ & $(2xy+w^2-z^2,z^2+x^2w+2yw)$   \\
$M^\sharp_{1,-1}$ & $(2xy+w^2-z^2,z^2+x^2(w-z)+2yw)$     \\
$I_{1,-1}$ & $(x(y-z)+w^3, xw+y(z-x))$   \\
\hline
\end{tabular}
\end{center}
\caption{Setting $k=-1$ in Wall's equations} \label{TabICIS}
\end{table}

\section{Invertible polynomials} \label{sect:Z2}
We recall some general definitions.

A complete intersection singularity in $\CC^n$ given by polynomial equations $f_1= \cdots = f_k=0$ is called {\em weighted homogeneous} if there are positive integers $w_1, \ldots, w_n$ (called {\em weights}) and $d_1, \ldots, d_k$ (called {\em degrees}) such that $f_j(\lambda^{w_1}x_1, \ldots , \lambda^{w_n}x_n)=\lambda^{d_j} f_j(x_1, \ldots, x_n)$ for $j=1, \ldots, k$ and for $\lambda \in \CC^\ast$. We call $(w_1, \ldots, w_n; d_1, \ldots, d_k)$ a {\em system of weights}.

A weighted homogeneous polynomial $f(x_1, \ldots, x_n)$ is called {\em invertible} if it can be written
\[ f(x_1, \ldots , x_n)=\sum_{i=1}^n a_i \prod_{j=1}^n x_j^{E_{ij}},
\]
where $a_i \in \CC^\ast$, $E_{ij}$ are non-negative integers, and the $n \times n$-matrix $E:=(E_{ij})$ is invertible over $\QQ$.

An invertible polynomial is called {\em non-degenerate} if it has an isolated singularity at the origin.

Let $f$ be an invertible polynomial given as above. By rescaling of the variables, one can assume that $a_i=1$ for $i=1, \ldots, n$. Moreover, we can assume that $\det E >0$.

The {\em Berglund-H\"ubsch transpose} \cite{BH} $\widetilde{f}$ of $f$ is defined by the transpose matrix $E^T$ of $E$, i.e.\ 
\[ \widetilde{f}(x_1, \ldots , x_n)=\sum_{i=1}^n a_i \prod_{j=1}^n x_j^{E_{ij}}.
\]

Let $f(x_1, \ldots, x_n)$ be an invertible polynomial.
The {\em canonical system of weights} $W_f$ of $f$ is the system of weights 
$(w_1,\dots ,w_n;d)$ given by the unique solution of the equation
\begin{equation*}
E
\begin{pmatrix}
w_1\\
\vdots\\
w_n
\end{pmatrix}
={\rm det}(E)
\begin{pmatrix}
1\\
\vdots\\
1
\end{pmatrix}
,\quad 
d:={\rm det}(E).
\end{equation*} 
We define
\[ q_1:= \frac{w_1}{d} , \ldots , q_n:= \frac{w_n}{d}.
\]

The {\em maximal group of diagonal symmetries} of $f$ is the group
\[ G_f=\{(\lambda_1, \ldots, \lambda_n)\in (\CC^*)^n: 
f(\lambda_1 x_1, \ldots, \lambda_n x_n)= f(x_1, \ldots, x_n)\}. \] 
It always contains the exponential grading operator 
\[
g_0:=(e^{2 \pi i q_1}, \ldots , e^{2 \pi i q_n}).
\]
Denote by $G_0$ the subgroup of $G_f$ generated by $g_0$.

By \cite{BHe} (see also \cite[Proposition~2]{EGBLMS}), ${\rm Hom}(G_f, \CC^\ast)$ is isomorphic to $G_{\widetilde{f}}$. For a subgroup $G \subset G_f$, Berglund and Henningson \cite{BHe} defined its {\em dual group} $\widetilde{G}$ by
\[ \widetilde{G} := {\rm Hom}(G_f/G, \CC^\ast).
\]
Note that $|\widetilde{G}_0|=c_f$, see \cite[Proposition~3.1]{ET2}.

\section{Wall's reduction}
Let $(X,0)$ be an ICIS in $\CC^4$ given by an equation $F=0$ where
\[ F(x,y,z,w)=(xy-a(z,w),yb(z,w)+c(x,z,w))
\]
where $a(z,w)$ and $c(x,z,w)$ are polynomials of degree $\geq  2$, $b(z,w)$ is a polynomial of degree $\geq 1$, and $x, b(z,w)$ form a regular sequence in $\CC[x,z,w]$.
Then we can consider the reduction 
\begin{equation*} \label{eq:matfac}
L_yF(x,z,w)=xc(x,z,w)+a(z,w)b(z,w)
\end{equation*}
of \cite{WPLMS} corresponding to the variable $y$. This means that we eliminate the variable $y$ to get the equation of a hypersurface singularity in $\CC^3$. Geometrically, this elimination corresponds to the projection along the $y$-axis on the coordinate space of the remaining variables $x,z,w$. It is proved in \cite[Theorem~7.9]{WPLMS}, for the case $b(z,w)=z$, that the Milnor number increases by one. 

In \cite{ET4}, we considered certain polynomials of the form
\[ 
f(x,z,w) = xc(x,z,w)+a(z,w)b(z,w)
\]
with the conditions on $a(z,w)$, $b(z,w)$, and $c(x,z,w)$ as above
and associated a complete intersection singularity to a graded matrix factorization  of size two of $f$. We showed that, in this way, we get an inverse to Wall's reduction. More precisely, a matrix factorization of $f$ is given by two matrices
\[ q_0 = \begin{pmatrix} a(z,w) & -x\\  c(x,z,w) & b(z,w) \end{pmatrix} \mbox{ and }  q_1= \begin{pmatrix} b(z,w) & x \\  -c(x,z,w) & a(z,w) \end{pmatrix}
\]
such that
\[ q_0q_1=q_1q_0= \begin{pmatrix} f & 0\\  0& f \end{pmatrix}.
\]
We associate to this the complete intersection singularity $(X_Q,0)$ given by 
\[ \mathbf{F}_Q(x,y,z,w)  =  (F_{Q,1}(x,y,z,w),F_{Q,2}(x,y,z,w)) := (a(z,w)-xy, c(x,z,w)+yb(z,w)).
\]
If $f=L_yF$, then we obtain back $(X_Q,0)=(X,0)$.

According to \cite{ET4}, the quadrangle ICIS correspond to matrix factorizations of the quadrangle hypersurface singularities  $Q_{2,0}$, $S_{1,0}$, $S^\sharp_{1,0}$, and $U_{1,0}$. They are given by non-degenerate invertible polynomials with $[G_f:G_0]=2$. In \cite[Proposition~1]{ET3}, we classified such polynomials. The coordinates are chosen so that the action of $\widetilde{G}_0= \ZZ/2\ZZ$ on $\widetilde{f}$ is given by $(x,y,z) \mapsto (-x,-y,z)$. In \cite[Proposition~2]{ET3}, we classified certain $4 \times 3$ matrices corresponding to such polynomials. Using \cite[Table~9 \& 10]{ET3}, this amounts to the list of Table~\ref{TabBimod}. The last column indicates the coefficients $a_1,a_2,a_3,a_4 \in \CC$ which are used in \cite{ET3} and will be used in Section~\ref{sect:virt}.
\begin{table}[h]
\begin{center}
\begin{tabular}{|c|c|c|c|c|}
\hline
Name & Type & $p_1,p_2(q_2),p_3(q_3)$ & $\ff$ & $a_1,a_2,a_3,a_4$  \\
\hline
$Q_{2,0}$ & ${\rm IV}_1$ & 3,12,24 & $a_1x^3+a_2xy^4+a_3yz^2+a_4x^2y^2$ & $1,1,1,-2$\\
$S_{1,0}$ & ${\rm IV}_2$ & 5,10,20 & $a_1x^5+a_2xy^2+a_3yz^2+a_4x^3y$ & $1,1,1,-2$\\
$S^\sharp_{1,0}$ & ${\rm IV}^\sharp_2$ & 5,10,20 & $a_1x^2z^2+a_2xy^2+a_3yz^2+a_4x^3y$ & $-1,1,1,-1$ \\
$U_{1,0}$ & ${\rm IV}^\sharp_2$ & 3,6,18 & $a_1xz^3+a_2xy^2+a_3yz^3+a_4x^2y$ & $-1,1,1,-1$ \\
\hline
\end{tabular}
\end{center}
\caption{Functions $\ff$ of 4 of the quadrangle hypersurface singularities} \label{TabBimod}
\end{table}

We now consider the matrix factorizations of the functions $\ff$ of Table~\ref{TabBimod}. They are given  in Table~\ref{TabEqu}, where we use suitable coordinates $(x,z,w)$ instead of $(x,y,z)$. 
\begin{table}[h]
\begin{center}
\begin{tabular}{|c|c|c|c|}
\hline
Name & Coord. change & $\ff$ & Name  \\
\hline
$Q_{2,0}$ & $(x,y,z) \mapsto (w,x,z)$  & $x(a_2wx^3+a_3z^2+a_4xw^2) + (a_1w) \cdot w^2$ &$J'_{2,0}$  \\
$S_{1,0}$ & $(x,y,z) \mapsto (x,z,w)$  &  $x(a_1x^4+a_2z^2+a_4x^2z) +(a_3z) \cdot w^2$  & $K'_{1,0}$  \\
$S_{1,0}$ & $(x,y,z) \mapsto (x,z,w)$  & $x(a_1x^4+a_2z^2+a_4x^2z) +(a_3w) \cdot zw$ & $L_{1,0}$  \\
$S^\sharp_{1,0}$ & $(x,y,z) \mapsto (x,z,w)$  & $x(a_1xw^2+a_2z^2+a_4x^2z) +(a_3z) \cdot w^2$  & $K^\flat_{1,0}$ \\
$S^\sharp_{1,0}$  & $(x,y,z) \mapsto (x,z,w)$ & $x(a_1xw^2+a_2z^2+a_4x^2z)  +(a_3w) \cdot zw$  & $L^\sharp_{1,0}$   \\
$U_{1,0}$  & $(x,y,z) \mapsto (w,x,z)$ &$x(a_2xw+a_3z^3+a_4w^2)+(a_1z) \cdot z^2w$  & $M_{1,0}$ \\
$U_{1,0}$  & $(x,y,z) \mapsto (z,w,x)$ & $x(a_1x^2z+a_3x^2w) + (a_2w+a_4z) \cdot zw$  & $M^\sharp_{1,0}$  \\
$U_{1,0}$ & $(x,y,z) \mapsto (x,z,w)$ & $x(a_1w^3+a_2z^2+a_4xz)+ (a_3z) \cdot w^3$  & $I_{1,0}$  \\ 
\hline
\end{tabular}
\end{center}
\caption{Matrix factorizations of the functions $\ff$} \label{TabEqu}
\end{table}
In the case $Q_{2,0}$, the matrix factorization from \cite[Table~2]{ET4}
\[ q_1= \begin{pmatrix} y^4 -(1+\lambda_4)xy^2+\lambda_4 x^2 & -y \\  z^2 & -x \end{pmatrix}
\]
is equivalent to the matrix factorization from Table~\ref{TabEqu} with $a_1=\lambda_4$, $a_2=-1$, $a_3=1$, $a_4=1+\lambda_4$, which is seen by adding the second column multiplied by $y^3-(1+\lambda_4)xy$ to the first column.

\section{An extension of the Berglund-H\"ubsch duality}
We shall now show that the duality between the quadrangle complete intersection singularities can be derived from the Berglund-H\"ubsch transposition of invertible polynomials in 4 variables. We use the procedure in \cite{ET3} to associate a weighted homogeneous non-invertible polynomial with 4 terms in 4 variables to each of the quadrangle complete intersection singularities. We consider the complete intersection singularities associated to the matrix factorizations in Table~\ref{TabEqu} defined by equations $(F_{Q,1},F_{Q,2}$, where we set $a_i=1$, $i=1, \ldots , 4$, and where we take a suitable order of the terms. Moreover, in the equation for $I_{1,0}$ we replace $XZ+YZ$ by $X^2+Y^2$. We also substitute temporarily the variables $x,y,z,w$ by capital letters $X,Y,Z,W$. We have the following 4 cases:
\begin{itemize}
\item[(a)] $F_{Q,1}(X,Y,Z,W)=XY-W^2$,
\item[(b)] $F_{Q,1}(X,Y,Z,W)=XY-W^3$,
\item[(c)] $F_{Q,1}(X,Y,Z,W)=XY-ZW$,
\item[(d)] $F_{Q,1}(X,Y,Z,W)=XY-Z^2W$.
\end{itemize}
We make the following coordinate substitutions in $F_{Q,2}(X,Y,Z,W)$:
\begin{eqnarray*}
{\rm (a)} &  XY-W^2: & X:=x^2w,  Y:=y^2w,  Z:=z, W:=xyw, \label{sub:a}  \\
{\rm (b)} & XY-W^3: & X:=x^6w^3, Y:=y^6w^3, Z:=z, W:=x^2y^2w^2. \label{sub:b} \\
{\rm (c)} & XY-ZW: & X:=xw, Y:=yz, Z:=xz, W:=yw, \label{sub:c} \\
{\rm (d)} & XY-Z^2W: & X:=y^2z^2, Y:=x^2w^2, Z:=xz, W:=y^2w^2, \label{sub:d}
\end{eqnarray*}
Then the polynomial $F_{Q,2}(X,Y,Z,W)$ is transformed to an invertible polynomial
\[
f(x,y,z,w)= \sum_{i=1}^4 x^{E_{i1}}y^{E_{i2}}z^{E_{i3}}w^{E_{i4}}
\]
for a $4 \times 4$-matrix $E$ of exponents. The corresponding polynomials are listed in Table~\ref{TabDual}. 

This procedure can be explained as follows.
We observe that the kernel of the matrix $E$ is generated by one of the following vectors:
\begin{itemize}
\item[(a),(b)] $(1,1,0,-2)^T$,
\item[(c),(d)] $(1,1,-1,-1)^T$,
\end{itemize}
Let $R:=\CC[x,y,z,w]$. There exists a $\ZZ$-graded structure on $R$ given by the respective $\CC^\ast$-action (here $\lambda \in \CC^\ast$):
\begin{eqnarray*}
{\rm (a), (b)} & & \lambda \ast (x,y,z,w) = (\lambda x, \lambda y, z , \lambda^{-2} w)  \\
{\rm (c), (d)} &  &  \lambda \ast (x,y,z,w) = (\lambda x, \lambda y, \lambda^{-1}z , \lambda^{-1} w) \\
\end{eqnarray*}
Let $R= \bigoplus_{i \in \ZZ} R_i$ be the decomposition of $R$ according to one of these $\ZZ$-gradings. The new coordinates $X,Y,Z,W$ are some invariant polynomials with respect to these actions and they satisfy the relation given by the corresponding first equation.

\begin{table}[h]
\begin{center}
\begin{tabular}{|c|c|c|c|}
\hline
Name & $(F_{Q,1},F_{Q,2})$  & $f$ & Dual \\
\hline
$J'_{2,0}$  & $(XY-W^2,X^3W+YW+Z^2+XW^2)$  & $x^7yw^4+xy^3w^2+z^2+x^4y^2w^3$& $J'_{2,0}$ \\
$K'_{1,0}$  & $(XY-W^2,X^4+Z^2+YZ+X^2Z)$  & $x^8w^4+z^2+y^2zw+x^4zw^2$ & $K'_{1,0}$\\
$K^\flat_{1,0}$ & $(XY-W^2,XW^2+Z^2+YZ+X^2Z)$ & $x^4y^2w^3+z^2+y^2zw+x^4zw^2$ & $L_{1,0}$ \\
$L_{1,0}$ & $(XY-ZW,X^4+Z^2+YW+X^2Z)$  & $x^4w^4+x^2z^2+y^2zw+x^3zw^2$ & $K^\flat_{1,0}$\\
$L^\sharp_{1,0}$ & $(XY-ZW,XW^2+Z^2+YW+X^2Z)$  & $xy^2w^3+x^2z^2+y^2zw+x^3zw^2$ & $L^\sharp_{1,0}$\\
$M_{1,0} $ & $(XY-Z^2W,Z^3+W^2+YZ+XW)$ & $x^3z^3+y^4w^4+x^3zw^2+y^4z^2w^2$ & $M_{1,0}$\\
$M^\sharp_{1,0}$ & $(XY-ZW,X^2W+YZ+YW+X^2Z)$ & $x^2yw^3+xyz^2+y^2zw+x^3zw^2$ & $M^\sharp_{1,0}$\\
$I_{1,0}$ & $(XY-W^3,X^2+Y^2+Z^2+W^3)$  & $x^{12}w^6+y^{12}w^6+z^2+x^6y^6w^6$  & $I_{1,0}$\\ 
\hline
\end{tabular}
\end{center}
\caption{Strange duality} \label{TabDual}
\end{table}

An inspection of Table~\ref{TabDual} shows that the Berglund-H\"ubsch transpose of the polynomial $f$ is either the polynomial $f$ itself or another polynomial appearing in the table. This leads to the indicated duality.

\section{Virtual isolated complete intersection singularities} \label{sect:virt}
We now derive the equations for the virtual singularities.  

In \cite[Section~4]{ET3}, we associated a polynomial $\hh$ to $\ff$, which defines the corresponding virtual bimodal hypersurface singularity. 
This is done as follows. We consider the polynomial $\ff(x,z,w)$ from Table~\ref{TabBimod} with the choice of coefficients  $a_1,a_2,a_3,a_4$ given in the last column. Then the corresponding equation defines a non-isolated singularity. We consider the cusp singularity
\[ \ff(x,z,w)-xzw
\]
and perform the coordinate change indicated in Table~\ref{Tab:ftoh}.
\begin{table}[h]
\begin{center}
\begin{tabular}{|c|c|c|c|}
\hline
Name & $\ff$ & Coord.~change & $\hh$ \\
\hline
$Q_{2,-1}$ & $w^3+x^4w+xz^2-2x^2w^2$ & $w \mapsto w+x^2$ & $w^3+x^2w^2+xz^2-x^3z$  \\
$S_{1,-1}$ & $x^5+xz^2+zw^2-2x^3z$ & $z \mapsto z+x^2$ & $xz^2+zw^2+x^2w^2-wx^3$ \\
$S^\sharp_{1,-1}$ & $-x^2w^2+xz^2+zw^2-x^3z$ & $z \mapsto z+x^2$ & $xz^2+x^3z+zw^2-x^3w$ \\
$U_{1,-1}$ & $-z^3w+x^2w+xz^3-xw^2$ & $x \mapsto x+w$ & $x^2w+xw^2+xz^3-zw^2$ \\
\hline
\end{tabular}
\end{center}
\caption{Functions $\hh$ of 4 of the quadrangle hypersurface singularities} \label{Tab:ftoh}
\end{table}
Then this polynomial is transformed to 
\[ \hh(x,z,w)-xzw,
\]
where the polynomial $\hh(x,z,w)$ is indicated in Table~\ref{Tab:ftoh}. The polynomial $\hh(x,z,w)$ has an isolated singularity at the origin, but also an additional critical point of type $A_1$ outside the origin. Moreover, if we consider the 1-parameter family $\hh(x,z,w)-t \cdot xzw$ for $t \in [0,1]$, then, for $t\neq 0,1$, the polynomial $\hh(x,z,w)-t \cdot xzw$ has two additional critical points of type $A_1$ outside the origin. One of them merges with the singularity of $\hh(x,z,w)$ at the origin for $t=0$ and the other one merges with the singularity of $\ff(x,z,w)-xzw$ at the origin for $t=1$.

\begin{example} Consider the case $Q_{2,0}$. Then
\begin{eqnarray*}
\lefteqn{\ff(x,z,w+x^2)-xz(w+x^2)} \\
 & = & (w+x^2)^3+x^4(w+x^2)+xz^2-2x^2(w+x^2)^2-xz(w+x^2)\\
 & = & w^3+x^2w^2+xz^2-x^3z-xzw \\
 &  = &  \hh(x,z,w)-xzw.
\end{eqnarray*}
The polynomial $\hh(x.z,w)$ has a singularity of Arnold type $Q_{12}$ at the origin.
On the other hand, for $t \neq 0$,
\begin{eqnarray*}
\lefteqn{\hh\left(x,z,w-\frac{1}{t}x^2\right)-t \cdot xz\left(w-\frac{1}{t}x^2\right)} \\
& = & \left(w-\frac{1}{t}x^2 \right)^3+x^2\left(w-\frac{1}{t}x^2 \right)^2+xz^2-x^3z    -t \cdot xz\left(w-\frac{1}{t}x^2 \right) \\
& = & w^3+ \left(3\frac{1}{t^2}-2\frac{1}{t}\right)x^4w + \left( \frac{1}{t^2}-\frac{1}{t^3} \right)x^6+xz^2+\left(1-3\frac{1}{t} \right)x^2w^2 - t \cdot xzw.
\end{eqnarray*}
Using the proof of \cite[Theorem~10]{ET1}, one can show that, for $t \neq 1$, this is a cusp singularity of type $T_{3,3,6}$. For $t=1$, it is a cusp singularity of type $T_{3,3,7}$. Using this, one can check the above statements.
\end{example}

Now we are looking at possible matrix factorizations of the polynomials $\hh$ of Table~\ref{Tab:ftoh}. They are listed together with the corresponding isolated complete intersection singularities  in Table~\ref{Tab0->Virt}. The corresponding isolated complete intersection singularity is denoted by $\HH=(\hh_1, \hh_2)$. The resulting singularities defined  by $\HH=(\hh_1, \hh_2)$ are called the {\em virtual singularities} and they are denoted by replacing the index 0 by $-1$.
\begin{table}[h]
\begin{center}
\begin{tabular}{|c|c|c|c|}
\hline
Name  & Matrix factorization of $\hh$  & $(\hh_1,\hh_2)$  & Virtual \\
\hline
$Q_{2,-1}$ & $x(-x^2z+z^2+xw^2)+w \cdot w^2$  & $(xy-w^2,-x^2z+yw+z^2+xw^2)$  & $J'_{2,-1}$\\
$S_{1,-1}$  & $x(-x^2w+z^2+xw^2)+z \cdot w^2$ & $(xy-w^2, -x^2w+z^2+yz+xw^2)$  & $K'_{1,-1}$ \\
$S^\sharp_{1,-1}$ & $x(-x^2w+z^2+x^2z)+z \cdot w^2$  & $(xy-w^2, -x^2w+z^2+yz+x^2z)$ & $K^\flat_{1,-1}$  \\
$S_{1,-1}$ & $x(-x^2w+z^2+xw^2)+zw \cdot w$ & $(xy-zw, -x^2w+z^2+yw+xw^2)$ & $L_{1,-1}$ \\
$S^\sharp_{1,-1}$ & $x(-x^2w+z^2+x^2z)+zw \cdot w$ & $(xy-zw, -x^2w+z^2+yw+x^2z)$ & $L^\sharp_{1,-1}$   \\
$U_{1,-1}$  & $x(xw+z^3+w^2)-z \cdot w^2$ & $(xy-w^2,-yz+xw+z^3+w^2)$ & $M_{1,-1}$ \\
$U_{1,-1}$ & $x(-z^2+x^2w)+zw \cdot (z+w)$ & $(xy-zw,-z^2+yw+x^2w+yz)$ & $M^\sharp_{1,-1}$ \\
$U_{1,-1}$  & $x(-xw+z^2+xz)+ z \cdot w^3$  & $(xy-w^3, -xw+z^2+yz+xz)$ & $I_{1,-1}$   \\ 
\hline
\end{tabular}
\end{center}
\caption{Virtual singularities} \label{Tab0->Virt}
\end{table}
There is another matrix factorization in the case $U_{1,-1}$, namely
\[
x(xw+z^3+w^2)-zw \cdot w .
\]
It is equivalent to the matrix factorization corresponding to $I_{1,-1}$.

Let 
\[ \HH(x,y,z,w) = (\hh_1(x,y,z,w), \hh_2(x,y,z,w)) 
\]
with
\[ \hh_2(x,y,z,w)=\sum_{i=1}^4 a_i x^{A_{i1}}y^{A_{i2}}z^{A_{i3}}w^{A_{i4}}
\]
 be the equations defining a virtual singularity and let ${\rm Supp}(\hh_2)=\{ (A_{i1}, A_{i2}, A_{i3}, A_{i4}) \in \ZZ^4 \, | \, i=1, \ldots ,4\}$.
Let $\Gamma_\infty(\hh_2)$ be the Newton polygon of $\hh_2$ at infinity \cite{Ko}, i.e.\ $\Gamma_\infty(\hh_2)$ is the convex closure in $\RR^4$ of ${\rm Supp}(\hh_2) \cup \{ 0 \}$. The Newton polygon $\Gamma_\infty(\hh_2)$ has two faces which do not contain the origin.  
Call these faces $\Sigma_1$ and $\Sigma_2$. Let $I_k:= \{i \in \{1, \ldots, 4\} \, | \, (A_{i1}, A_{i2}, A_{i3}, A_{i4}) \in \Sigma_k \}$, $k=1,2$, and let
\[
\hh_{2,k}= \sum_{i \in I_k} a_i x^{A_{i1}}y^{A_{i2}}z^{A_{i3}}w^{A_{i4}}.
\]
Then $(\hh_1,\hh_{2,k})$ defines a non-isolated weighted homogeneous complete intersection singularity.  The polynomials $\hh_1$ and $\hh_2$ and their systems of weights are listed in Table~\ref{TabVirt}.  
\begin{table}[h]
\begin{center}
\begin{tabular}{|c|c|c|c|c|}
\hline
Name   & $\hh_{2,1}$  & Weights & $\hh_{2,2}$ & Weights\\
\hline
$J'_{2,-1}$  & $-x^2z+z^2+xw^2$ & $2,4,4,3;6,8$  &$z^2+xw^2+yw$ & $2,6,5,4;8,10$ \\
$K'_{1,-1}$   & $-x^2w+z^2+xw^2$  & $2,2,3,2;4,6$ & $z^2+xw^2+yz$ & $2,4,4,3;6,8$ \\
$K^\flat_{1,-1}$   & $-x^2w+x^2z+yz$ & $2,4,3,3;6,7$ & $x^2z+yz+z^2$ & $2,4,4,3;6,8$\\
$L_{1,-1}$    &  $-x^2w+z^2+xw^2$  & $2,3,3,2;5,6$ & $z^2+xw^2+yw$  & $2,5,4,3;7,8$ \\
$L^\sharp_{1,-1}$   & $-x^2w+x^2z+yw$ & $2,4,3,3;6,7$ & $x^2z+yw+z^2$ & $2,5,4,3;7,8$ \\
$M_{1,-1}$   &  $xw+z^3+w^2$ & $3,3,2,3;6,6$ &  $z^3+w^2-yz$ & $2,4,2,3;6,6$ \\
$M^\sharp_{1,-1}$   & $-z^2+x^2w+yz$ & $2,3,3,2;5,6$ & $x^2w+yz+yw$ & $2,4,3,3;6,7$ \\
$I_{1,-1}$   & $-xw+yz+xz$ & $3,3,2,2;6,5$ & $yz+xz+z^2$ & $3,3,3,2;6,6$ \\ 
\hline
\end{tabular}
\end{center}
\caption{Decomposition of equations} \label{TabVirt}
\end{table}

\section{Dolgachev numbers}

We shall now define Dolgachev numbers for our virtual singularities.

The Dolgachev numbers of the virtual singularity $(\hh_1,\hh_2)$ are defined in a similar way as \cite[Section~5]{ET3}. Let $i=1,2$ and let
$V_i:= \{ (x,y,z,w) \in \CC^4 \, | \, \hh_1(x,y,z,w)=0, \ \hh_{2,i}(x,y,z,w)=0\}$. 
We consider the $\CC^\ast$-action on $V_i$ given by the system of weights of $(\hh_1,\hh_{2,i})$ (see Table~\ref{TabVirt}). We consider exceptional orbits (i.e.\ orbits with a non-trivial isotropy group) of this action. We distinguish between three cases:
\begin{itemize}
\item[(A)] $V_i$ contains a linear subspace $L$ of $\CC^4$ of codimension 2 obtained by setting two coordinates to be zero.
\item[(B)] $V_i=U \cup U'$, where 
\[
(\hh_1(x,y,z,w),\hh_{2,i}(x,y,z,w))=(g_1(x,y,w), zg_2(x,y,z)),
\]
\[
U=\{ (x,y,z,w) \in \CC^4 \, | \, g_1(x,y,w)=z=0 \}, 
\]
\[
U'= \{ (x,y,z,w) \in \CC^4 \, | \, g_1(x,y,w)=g_2(x,y,z)=0 \}.
\]
\item[(C)] $V_i$ is not of the form of (A) or (B).
\end{itemize}
In case (A) we consider those exceptional orbits which are not contained in $L$. In case (B) we consider those exceptional orbits which are not contained in $U$. In case (C) we consider those exceptional orbits which do not coincide with the singular locus of $V_i$. We call these the {\em principal} orbits. It turns out that in all cases we have exactly two principal orbits. 

\begin{definition} The {\em Dolgachev numbers} of the virtual singularity $(\hh_1,\hh_2)$ are the numbers $\alpha_1, \alpha_2; \alpha_3, \alpha_4$ where $\alpha_1, \alpha_2$ and $\alpha_3, \alpha_4$ are the orders of the isotropy groups of the principal exceptional orbits of $(\hh_1,\hh_{2,1})$ and $(\hh_1,\hh_{2,2})$ respectively.
\end{definition}

The Dolgachev  numbers of the virtual singularities are computed as follows. The two pairs of polynomials $(\hh_1, \hh_{2,i})$, $i=1,2$, of Table~\ref{TabVirt} define non-isolated weighted homogeneous complete intersection singularities of certain types. The systems of weights correspond to the five quadrangle ICIS and three elliptic complete intersection singularities as considered by Wall \cite{W2}. We indicate the notation of Wall \cite{W2}  in Table~\ref{TabDol}. The corresponding orbifold curves have genus zero. We list the orders of the isotropy groups of the exceptional orbits of these ICIS in this table (see also \cite{ERIMS}). Some of them correspond to the orders of the isotropy groups of the exceptional orbits for the non-isolated singularities given by the pairs $(\hh_1, \hh_{2,i})$, $i=1,2$. Those ones which do not occur are stroken out. The orders of the isotropy groups of the principal orbits are indicated in bold face. We also indicate for each pair which of the corresponding cases (A), (B), or (C) applies. An exceptional orbit which coincides with the singular locus is marked by $^*$. The Dolgachev numbers $\alpha_1, \alpha_2; \alpha_3, \alpha_4$ of the virtual singularities are indicated in the last column.
\begin{table}[h]
\begin{center}
\begin{tabular}{|c|c|cc|c|cc|c|}
\hline
 Name & $(\hh_1,\hh_{2,1})$  & & orbits  &  $(\hh_1,\hh_{2,2})$ & & orbits & $\alpha_1,\alpha_2;\alpha_3, \alpha_4$ \\
\hline
$J'_{2,-1}$   & $K'_{1,0}$  & (C) &  {\bf 2,2},$4^*,\not 4$  & $J'_{2,0}$  & (C) & $\not 2,2^*$,{\bf 2,6} & 2,2;2,6 \\
$K'_{1,-1}$  & $\delta 1$ & (C)  & {\bf 2,2},$2^*,\not 2, \not 2$  &  $K'_{1,0}$ & (C) & $\not 2,2^*$,{\bf 4,4} & 2,2;4,4\\
$K^\flat_{1,-1}$  & $M_{1,0}$  & (C) & {\bf 2},$\not 3,3^*$,{\bf 4} & $K'_{1,0}$ & (B)   & {\bf 2},2,4,{\bf 4} & 2,4;2,4\\
$L_{1,-1}$  &  $\alpha 1^{(2)}$ & (A) & {\bf 2,2},2,$3^*$,$\not 3$ & $L_{1,0}$  & (C) & $\not 2,2^*$,{\bf 3,5} & 2,2;3,5\\
$L^\sharp_{1,-1}$  & $M_{1,0}$  & (A) & {\bf 2,3},$3^*$,4 & $L_{1,0}$ & (A) & {\bf 2},2,3,{\bf 5} & 2,3;2,5\\
$M_{1,-1}$   & $I_{1,0}$  & (C) & {\bf 3,3},$3^*,\not 3$ &  $\alpha 1^{(1)}$ & (C) & $\not 2, \not 2,2^*$,{\bf 2,4} & 3,3;2,4\\
$M^\sharp_{1,-1}$  & $\alpha 1^{(2)}$  & (A) & {\bf 2},$\not 2,2^*$,3,{\bf 3} &  $M_{1,0}$ & (A) & 2,3,${\bf 3}^*$,{\bf 4} & 2,3;3,4\\
$I_{1,-1}$ & $\alpha 1^{(2)}$ & (C) & $\not 2,\not 2,2^*$,{\bf 3,3} &  $I_{1,0}$ & (B) & 3,3,{\bf 3,3} & 3,3;3,3 \\
\hline
\end{tabular}
\end{center}
\caption{Dolgachev numbers} \label{TabDol}
\end{table}

\begin{example} 

(a) We consider the singularity $K^\flat_{1,-1}$. We have $(\hh_1, \hh_{2,1})=(xy-w^2, -x^2w+x^2z+yz)$ with the system of weights $(2,4,3,3;6,7)$. We are in case (C). The exceptional orbits are 
\begin{center}
\begin{tabular}{cl}
$x=y=w=0$ & singular line, order of isotropy group: 3\\
$y=z=w=0$ & order of isotropy group: 2\\
$x=z=w=0$ & order of isotropy group: 4
\end{tabular}
\end{center}
This gives $(\alpha_1,\alpha_2)=(2,4)$.

(b)  We again consider the singularity $K^\flat_{1,-1}$, but now $(\hh_1, \hh_{2,2})=(xy-w^2, x^2z+yz+z^2)$ with the system of weights $(2,4,4,3;6,8)$. Here we are in case (B). The singular locus is the curve $z=x^2+y=x^3+w^2=0$ with trivial isotropy group. It is contained in $U=\{ xy-w^2=z=0 \}$. The exceptional orbits contained in $U$ are
\begin{center}
\begin{tabular}{cl}
$y=z=w=0$ & order of isotropy group: 2\\
$x=z=w=0$ & order of isotropy group: 4
\end{tabular}
\end{center}
The exceptional orbits not contained in $U$ are
\begin{center}
\begin{tabular}{cl}
$y=w=x^2+z=0$ & order of isotropy group: 2\\
$x=w=y+z=0$ & order of isotropy group: 4
\end{tabular}
\end{center}
This gives $(\alpha_3,\alpha_4)=(2,4)$.

(c) We consider the case $L_{1,-1}$ with $(\hh_1,\hh_{2,1})=(xy-zw,-x^2w+xw^2+z^2)$. The system of weights is $(2,3,3,2;5,6)$. Here $V_1$ contains the hyperplane $L=\{x=z=0\}$, so we are in case (A). The exceptional orbits contained in $L$ are
\begin{center}
\begin{tabular}{cl}
$x=z=w=0$ & singular line, order of isotropy group: 3\\
$x=y=z=0$ & order of isotropy group: 2
\end{tabular}
\end{center}
The exceptional orbits not contained in $L$ are
\begin{center}
\begin{tabular}{cl}
$y=z=w=0$ & order of isotropy group: 2\\
$x-w=y=z=0$ & order of isotropy group: 2
\end{tabular}
\end{center}
This gives $(\alpha_1,\alpha_2)=(2,2)$.
\end{example}

\begin{remark} Using the primary decomposition algorithm of the computer algebra software {\sc Singular} \cite{DGPS}, one can show that, for each pair $(\hh_1,\hh_{2,i})$ where we have case (A), the subspace $L$ is an irreducible component of $V_i$. If one removes this component $L$ in case (A), the component $U$ in case (B), and the point corresponding to the singular line in case (C), one gets $\PP^1_{\alpha_{2i-1},\alpha_{2i}}$ with one point removed. Here $\PP^1_{\alpha_{2i-1},\alpha_{2i}}$ denotes the complex projective line with two orbifold points with singularities $\ZZ/\alpha_j\ZZ$, $j=2i-1,2i$.
\end{remark}
\section{Gabrielov numbers}

We now want to define Gabrielov numbers. They will be defined as in \cite[Section~5]{ET3}. For this purpose, we consider the pairs $\HH=(\hh_1,\hh_2)$ of polynomials of Table~\ref{Tab0->Virt}. Here the first three cases  $J'_{2,-1}$, $K'_{1,-1}$, and $K^\flat_{1,-1}$ are suspensions of the curve singularities  $J_{2,-1}$, $K_{1,-1}$, and $K^\sharp_{1,-1}$. In these cases, we consider as the first polynomial $\hh_1(x,y,z,w):=xy-w^2-z^2$. Then we consider the complete intersection singularity $(X',0)$ defined by
\[ \left\{ \begin{array}{l} \hh_1(x,y,z,w), \\ \hh_2(x,y,z,w)-zw. \end{array} \right.
\]
As in \cite{ET1}, one can show that the singularity $(X',0)$ is $\mathcal{K}$-equivalent to the singularity defined by
\[ \left\{ \begin{array}{l} xy-z^{\gamma_1}-w^{\gamma_2}, \\ x^{\gamma_3}+y^{\gamma_4}-zw. \end{array} \right.
\]
This means that $(X',0)$ is a cusp singularity of type $T^2_{\gamma_1,\gamma_3,\gamma_2,\gamma_4}$ in the notation of \cite[3.1]{E1}.

\begin{definition} The {\em Gabrielov numbers} of the virtual singularity given by the pair $(\hh_1,\hh_2)$ are the numbers $\gamma_1,\gamma_2;\gamma_3,\gamma_4$.
\end{definition}

\begin{proposition} The Gabrielov numbers of the virtual quadrangle complete intersection singularities are given by Table~\ref{TabBiGab}.
\end{proposition}

\begin{proof} We consider Wall's reduction of the virtual quadrangle singularities according to Table~\ref{Tab0->Virt}. The cusp singularity $\HH':=(\hh_1(x,y,z,w), \hh_2(x,y,z,w)-zw)$ corresponds to the hypersurface cusp singularity $L_y\HH'=\hh(x,z,w) -xzw$.  In all cases except $I_{1,-1}$, by transformations similar to the transformations in \cite{ET3}, we obtain the indicated Gabrielov numbers. More precisely, for a suitable permutation $\sigma: \{1,2,3,4\} \to \{1,2,3,4\}$, the Gabrielov numbers satisfy $(\gamma_{\sigma(1)}, \gamma_{\sigma(2)},\gamma_{\sigma(3)},\gamma_{\sigma(4)})=(2, \widetilde{\gamma}_1-1,\widetilde{\gamma}_2-1,\widetilde{\gamma}_3-1)$, where $(\widetilde{\gamma}_1,\widetilde{\gamma}_2,\widetilde{\gamma}_3)$ are the Gabrielov numbers of the corresponding virtual hypersurface singularity.

In the case $I_{1,-1}$, we indicate the claimed $\mathcal{K}$-equivalence. We add the first polynomial $\hh_1(x,y,z,w)$ to the second one $\hh_2(x;y,z,w)-zw$ and obtain
\begin{equation} (xy-w^3, xy-w^3-xw+z^2+yz+xz-zw). \label{eq:I1}
\end{equation}
By the transformation $w \mapsto w+y$, this is transformed to 
\begin{equation}
(xy-w^3-p_1(y,w),-y^3-w^3-q_1(y,w)-xw+z^2+xz-zw), \label{eq:I2}
\end{equation}
where $p_1(y,w)$ and $q_1(y,w)$ are certain polynomials of degree 3 in the variables $y$ and $w$. Using \cite[Lemma~7.3]{a:1} and the fact that $y$ divides $p_1(y,w)$, one can get rid of the polynomial $p_1(y,w)$ with the help of the term $xy$. Similarly, one can get rid of the polynomial $q_1(y,w)$ in the second equation with the help of the term $zw$. Now we apply the transformation $w \mapsto w+x$. Then the pair (\ref{eq:I2}) gets
\begin{equation}
(xy-w^3-x^3-p_2(x,w), -x^3-y^3 -q_2(x,w)-xw-x^2+z^2-zw), \label{eq:I3}
\end{equation}
where $p_2(x,w)$ and $q_2(x,w)$ are again polynomials which can be removed. Applying the transformation $x \mapsto x+z$, one gets
\begin{equation}
(xy-w^3-z^3 -p_3(x,z), -x^3-y^3-q_3(x,z)-xw-x^2-2xz-2zw), \label{eq:I4}
\end{equation}
again with certain removable polynomials $p_3(x,z)$ and $q_3(x,z)$.
Finally, by the transformation $z \mapsto \frac{1}{2}(z-x)$ followed by $w \mapsto w-x$ and rescaling, we obtain
\begin{equation}
(xy-z^3-w^3+p_4(x,w), x^3+y^3-zw),
\end{equation}
again with a removable polynomial $p_4(x,w)$.
\end{proof}

By \cite{E1}, one can compute Coxeter-Dynkin diagrams of the (global) singularities defined by $\HH=(\hh_1,\hh_2)$. Let $X^{(1)}:=\{(x,y,z,w) \in \CC^4 \, | \, \hh_1(x,y,z,w)=0\}$ and consider the function $\hh_2: X^{(1)} \to \CC$. It has besides the origin one or two additional critical points which are of type $A_1$. The singularity at the origin is indicated in Table~\ref{TabBiGab}. We define the {\em Milnor number} of the virtual singularity by the sum of the Milnor numbers of the singular points. It is equal to 12 in all cases. 

One can compute that there exists a (strongly) distinguished basis of thimbles $(e_1, \ldots , e_{\mu+1}) = (e_j^r \, | \, 1 \leq j \leq 8, 1 \leq r \leq M_j)$, where the intersection matrix of $(e_1, \ldots, e_8)=(e_1^1, \ldots , e_8^1)$ coincides with the  intersection matrix of the system $(\widehat{\delta}'_1, \ldots , \widehat{\delta}'_8)$ of \cite[Sect.~2.3]{E1}, the numbers $M_1, M_3, M_8$ are equal to one and the other numbers $M_2, M_4, M_5, M_6, M_7$ are indicated in Table~\ref{TabBiGab}, and the intersection matrix of $(e_j^r \, | \, 1 \leq j \leq 8, 1 \leq r \leq M_j)$ is computed according to \cite[Theorem~2.2.3]{E1}. By the proof of \cite[Proposition~3.6.1]{E1}, one can transform these bases to (strongly) distinguished bases of thimbles with  Coxeter-Dynkin diagrams of the form $\Pi_{\gamma_1,\gamma_2,\gamma_3,\gamma_4}$ of Fig.~\ref{FigPipqrs}, where $\gamma_1,\gamma_2;\gamma_3,\gamma_4$ are the Gabrielov numbers of the virtual singularity.

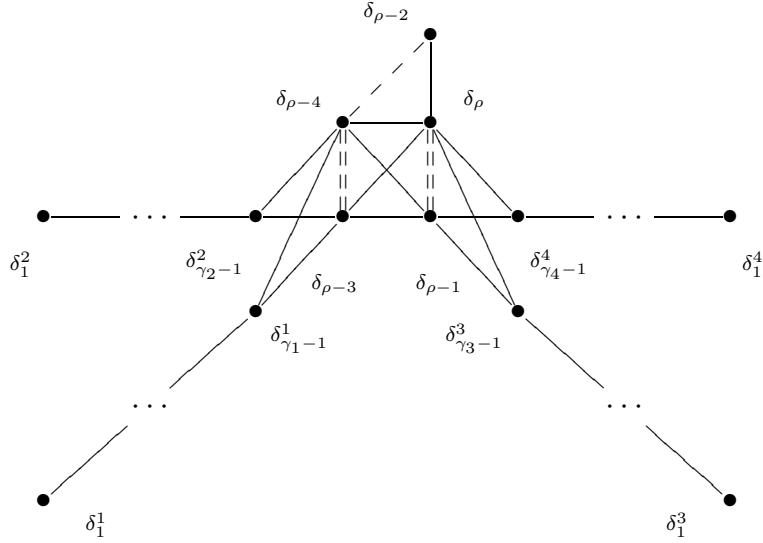
\begin{figure}
$$
\xymatrix{ 
 & & & & *{\bullet} \ar@{-}[d] \ar@{--}[dl] \ar@{}_{\delta_{\rho-2}}[l]  &  & & \\
 & & & *{\bullet} \ar@{-}[r] \ar@{==}[d] \ar@{-}[dr]  \ar@{-}[ldd] \ar@{}_{\delta_{\rho-4}}[l]
 & *{\bullet} \ar@{==}[d] \ar@{-}[dr]  \ar@{-}[rdd] \ar@{}^{\delta_{\rho}}[r]& & &  \\
 *{\bullet} \ar@{-}[r] \ar@{}_{\delta^2_1}[d]  & {\cdots} \ar@{-}[r]  & *{\bullet} \ar@{-}[r] \ar@{-}[ur]   \ar@{}_{\delta^2_{\gamma_2-1}}[d] & *{\bullet} \ar@{-}[dl] \ar@{-}[r] \ar@{-}[ur] \ar@{}^{\delta_{\rho-3}}[ld] & *{\bullet} \ar@{-}[r] \ar@{-}[dr] \ar@{}_{\delta_{\rho-1}}[rd] & *{\bullet} \ar@{-}[r]  \ar@{}^{\delta^4_{\gamma_4-1}}[d]  & {\cdots} \ar@{-}[r]  &*{\bullet} \ar@{}^{\delta^4_1}[d]   \\
& &  *{\bullet} \ar@{-}[dl] \ar@{}_{\delta^1_{\gamma_1-1}}[r]  & & & *{\bullet} \ar@{-}[dr] \ar@{}^{\delta^3_{\gamma_3-1}}[l] & &  \\
 & {\cdots} \ar@{-}[dl] & & & & & {\cdots} \ar@{-}[dr] & \\
*{\bullet}  \ar@{}_{\delta^1_1}[r] & & & & & & & *{\bullet}  
\ar@{}^{\delta^3_1}[l]
  } 
$$
\caption{The graph $\Pi_{\gamma_1,\gamma_2,\gamma_3,\gamma_4}$} \label{FigPipqrs}
\end{figure}

\begin{table}[h]
\begin{center}
\begin{tabular}{|c|c|ccccc|c|}
\hline
Virtual  & Germ at 0 &  $M_2$ & $M_4$ & $M_5$ & $M_6$ & $M_7$ & $\gamma_1, \gamma_2; \gamma_3, \gamma_4$  \\
\hline
$J'_{2,-1}$ & $J'_{11}$ & 3 & 1 & 1 & 1 & $3+1$ & $2,2;2,5+1$  \\
$K'_{1,-1}$ & $K'_{10}$ & 2 & $2+1$ & $2+1$ &1 & 1   & $2,2;3+1,3+1$ \\
$K^\flat_{1,-1}$ & $K'_{11}$ & $2+1$ & 2 & 3 & 1 & 1  & $2,2;3,4+1$ \\
$L_{1,-1}$ & $L_{10}$   & 2 & $2+1$  & 1 & 1 & $2+1$  & $2,3+1;2,3+1$    \\
$L^\sharp_{1,-1}$ & $L_{11}$  & $2+1$ & 3 & 1 & 1 & 2  & $2,3;2,4+1$  \\
$M_{1,-1}$  & $J'_{10}$ & 3 & $1+1$ & $1+1$ & 1 & 2  & $2+1,2+1;2,4$  \\
$M^\sharp_{1,-1}$  & $M_{11}$& 2 & $2+1$ & 2 & 1 & 2 & $2,3;3,3+1$  \\
$I_{1,-1}$ & $M_{11}$ & 2 & 2 & 2 & $1+1$ & 2  & $2+1,3;3,3$    \\
\hline
\end{tabular}
\end{center}
\caption{Coxeter-Dynkin diagrams of virtual singularities} \label{TabBiGab}
\end{table}

\section{Strange duality}
We now consider the duality defined in Section~\ref{sect:Z2}.
We summarise the results on  the Dolgachev and Gabrielov numbers of the virtual singularities in Table~\ref{TabQmon}. From this table, we get the following result:

\begin{theorem} The Gabrielov numbers of a virtual quadrangle complete intersection singularity coincide with the Dolgachev numbers of the dual one, and vice versa.
\end{theorem}

For another feature of this duality, we have to introduce some notions.

Let $f_1$, \dots , $f_k$ be quasihomogeneous functions on $\CC^n$  of
degrees $d_1$, \dots , $d_k$ with respect to weights $w_1$, \dots, $w_n$. 
We suppose that the equations $f_1=f_2= \ldots = f_k=0$ define a complete intersection $X$ in
$\CC^n$. There is a natural $\CC^\ast$-action on the space $\CC^n$ defined by 
$\lambda\ast(x_1, \ldots, x_n)=
(\lambda^{w_1}x_1, \ldots, \lambda^{w_n}x_n)$, $\lambda \in \CC^\ast$.

Let
$A=\CC[x_1, \ldots , x_n]/(f_1, \ldots , f_k)$ be the coordinate ring of $X$. Then the $\CC^\ast$-action on $\CC^n$ induces a
natural grading $A=\oplus_{s=0}^\infty A_s$ on the ring $A$, where 
$$A_s := \{  g\in A \, | \, g(\lambda\ast (x_1, \ldots, x_n))=\lambda^s g(x_1, \ldots, x_n) \mbox{ for } \lambda \in \CC^\ast\}.
$$ 

We shall consider the Poincar\'e series $P_X(t)=\sum_{s=0}^\infty \dim A_s \cdot t^s$
of this graded algebra. One has
\begin{equation*}
P_X(t) = \frac{ \prod_{j=1}^k (1-t^{d_j}) }{\prod_{i=1}^n (1-t^{w_i})}.
\end{equation*}

For a map $\varphi: Z \to Z$ of a topological space $Z$,
the {\em zeta function}
is defined to be
$$\zeta_{\varphi}(t)=\prod_{p\ge0}
\left\{\det \left( \mbox{id} -t\cdot {\varphi}_\ast\vert_{H_p(Z;\CC)}\right)\right\}^{(-1)^p}.$$
If, in the definition, we use the actions of the operators ${\varphi}_\ast$ on the reduced homology groups
$\overline{H}_p(Z;\ZZ)$, we get the {\em reduced} zeta function
$$
\overline\zeta_{\varphi}(t)  =  \frac{\zeta_{\varphi}(t)}{(1-t)}. 
$$

For $0 \leq j \leq k$, let $X^{(j)}$ be the complete intersection given by the equations $f_1=
\ldots = f_j=0$ ($X^{(0)}=\CC^n$, $X^{(k)}=X$). The restriction of the function $f_j$ ($j=1,
\ldots , k$) to the variety $X^{(j-1)}$ defines a locally trivial fibration $X^{(j-1)} \setminus
X^{(j)} \to \CC^\ast$. Let $V^{(j)} = f_j^{-1}(1) \cap X^{(j-1)}$ be the typical fibre (Milnor fibre) of this
fibration. Note that it is not necessarily smooth. There is a monodromy transformation $\varphi^{(j)} : V^{(j)} \to V^{(j)}$  on it.
Let
$$\overline\zeta_{X,j}(t) := \overline\zeta_{{\varphi}^{(j)}}(t).$$
One can show that $({\varphi}^{(j)}_\ast)^{d_j} = \mbox{id}$ and therefore 
$\overline\zeta_{X,j}(t)$ can be written in the form
$$
\prod_{\ell\vert d_j}(1-t^\ell)^{\alpha_\ell}, \ \alpha_\ell\in\ZZ.
$$
Following K.~Saito \cite{S1, S2}, we define the Saito dual to $\overline\zeta_{X,j}(t)$ to be the
rational function 
$$
\overline\zeta_{X,j}^\ast(t)=\prod_{m\vert d_j}(1-t^m)^{-\alpha_{(d_j/m)}}
$$
(note that different degrees $d_j$ are used for different $j$).

Let $Y^{(k)}=(X^{(k)} \setminus\{0\})/\CC^\ast$ be the space of orbits of the
$\CC^\ast$-action on $X^{(k)}\setminus\{0\}$ and
$Y^{(k)}_m$ be the set of orbits for which the isotropy group is the
cyclic group of order $m$. For a topological space $Z$, denote by $\chi(Z)$ its Euler characteristic. Define
$$\mbox{Or}_X(t) := \prod_{m \geq 1} (1-t^m)^{\chi(Y_m^{(k)})}.$$

Now let $X$ be an ICIS in $\CC^4$ defined by two polynomial equations $f_1=f_2=0$ and assume that both $X^{(1)}=f_1^{-1}(0)$ and $X^{(2)}=f_2^{-1}(0)$ have isolated singularities at the origin. Moreover, assume that $X^{(1)}$ has a singularity of type $A_1$. Consider the mapping $F:=(f_1,f_2) : \CC^4 \to \CC^2$. Let $C_F$ be the critical locus of $F$ and $D_F=F(C_F)$. The mapping $F|_{\CC^4-F^{-1}(D_F)} : \CC^4-F^{-1}(D_F) \to \CC^2 -D_F$ defines a locally trivial fibration. Assume that $(1,1) \not\in D_F$. Let $V^{(1)} = f_1^{-1}(1)$ and $V_2= f_2^{-1}(1) \cap V^{(1)}$. (Note that $V_2 \neq V^{(2)}$ but $V_2$ and $V^{(2)}$ are homeomorphic to each other.) Then $V_2 \subset V^{(1)}$ and the monodromy transformation $\varphi^{(1)} : V^{(1)} \to V^{(1)}$ induces a relative monodromy operator $\widehat{\varphi}_\ast: H_3(V^{(1)},V_2; \ZZ) \to H_3(V^{(1)},V_2; \ZZ)$. Let 
$$\Delta_X(t) := \det \left(\mbox{id}-t\cdot \widehat{\varphi}_\ast\vert_{H_3(V^{(1)},V_2;\CC)}\right)$$
be the characteristic polynomial of this operator. 

\begin{proposition} \label{prop:deltazeta}
We have
\[ \Delta_X(t) = (1-t)^2 \prod_{j=1}^2 \overline{\zeta}_{X,j}(t).
\]
\end{proposition}

\begin{proof}
We have the following commutative diagram of split short exact sequences:
\[ 
\xymatrix{
0  \ar[r] & H_3(V^{(1)};\ZZ) \ar[r]\ar[d]_{\varphi_\ast^{(1)}}  & H_3(V^{(1)},V_2;\ZZ) \ar[r]\ar[d]_{\widehat{\varphi}_\ast} &  H_2(V_2;\ZZ) \ar[r]\ar[d]_{\varphi_\ast^{(2)}} & 0\\
0\ar[r] & H_3(V^{(1)};\ZZ) \ar[r] & H_3(V^{(1)},V_2;\ZZ) \ar[r] & H_2(V_2;\ZZ) \ar[r] & 0
}
\]
This shows that 
\begin{eqnarray*}
\Delta_X(t) & = &  \det \left(\mbox{id}-t\cdot \varphi_\ast^{(1)}\vert_{H_3(V^{(1)};\CC)}\right) \det \left(\mbox{id}-t\cdot \varphi_\ast^{(2)}\vert_{H_2(V_2;\CC)}\right) \\
& = & \overline{\zeta}_{X,1}(t)^{-1} \overline{\zeta}_{X,2}(t) = (1-t)^2 \prod_{j=1}^2 \overline{\zeta}_{X,j}(t) 
\end{eqnarray*}
since $\overline{\zeta}_{X,1}(t)=(1-t)^{-1}$. 
\end{proof}

Let $X$ be an ICIS with a Coxeter-Dynkin diagram of type $\Pi_{\gamma_1,\gamma_2,\gamma_3,\gamma_4}$. 
Then the polynomial $\Delta_X(t)$ is equal to the characteristic polynomial $\Delta(\Pi_{\gamma_1,\gamma_2,\gamma_3,\gamma_4})(t)$ of the Coxeter element corresponding to this Coxeter-Dynkin diagram. By \cite[Proposition 3.6.2]{E1}, we have
$$\Delta(\Pi_{\gamma_1,\gamma_2,\gamma_3,\gamma_4})(t) = (1-t)^2 \Delta(S_{\gamma_1,\gamma_2,\gamma_3,\gamma_4})(t)$$ 
where $\Delta(S_{\gamma_1,\gamma_2,\gamma_3,\gamma_4})(t)$ is the characteristic polynomial of the Coxeter element corresponding to the graph $S_{\gamma_1,\gamma_2,\gamma_3,\gamma_4}$ depicted in Fig.~\ref{FigSpqrs}. (Note that there is a slight mistake in the proof of \cite[Proposition 3.6.2]{E1} which was corrected in \cite{E3a}.)  By Proposition~\ref{prop:deltazeta} we get
\begin{equation}
\prod_{j=1}^2 \overline{\zeta}_{X,j}(t) = \Delta(S_{\gamma_1,\gamma_2,\gamma_3,\gamma_4})(t). \label{eq:piS}
\end{equation}
This also gives an interpretation of the characteristic polynomial of the Coxeter element $c^\flat$ considered in \cite{E3a}.
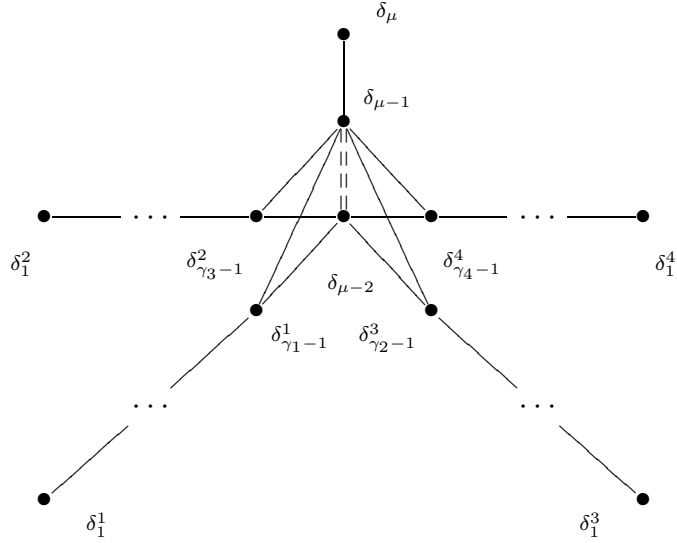
\begin{figure}
$$
\xymatrix{ 
 & & & *{\bullet} \ar@{-}[d] \ar@{}^{\delta_{\mu}}[r] & & & \\
 & & & *{\bullet} \ar@{==}[d] \ar@{-}[dr]  \ar@{-}[ldd] \ar@{-}[rdd] \ar@{}^{\delta_{\mu-1}}[r]
 & & &  \\
 *{\bullet} \ar@{-}[r] \ar@{}_{\delta^2_1}[d]  & {\cdots} \ar@{-}[r]  & *{\bullet} \ar@{-}[r] \ar@{-}[ur]   \ar@{}_{\delta^2_{\gamma_3-1}}[d] & *{\bullet} \ar@{-}[dl] \ar@{-}[dr] \ar@{-}[r] \ar@{}_{\delta_{\mu-2}}[dr] & *{\bullet} \ar@{-}[r]  \ar@{}^{\delta^4_{\gamma_4-1}}[d]  & {\cdots} \ar@{-}[r]  &*{\bullet} \ar@{}^{\delta^4_1}[d]   \\
 & &   *{\bullet} \ar@{-}[dl] \ar@{}_{\delta^1_{\gamma_1-1}}[r]  & & *{\bullet} \ar@{-}[dr] \ar@{}^{\delta^3_{\gamma_2-1}}[l]  & & \\
 & {\cdots} \ar@{-}[dl] & & & & {\cdots} \ar@{-}[dr] & \\
*{\bullet}  \ar@{}_{\delta^1_1}[r] & & & & & & *{\bullet}  
\ar@{}^{\delta^3_1}[l]
  } 
$$
\caption{The graph $S_{\gamma_1,\gamma_2,\gamma_3,\gamma_4}$} \label{FigSpqrs}
\end{figure}

A $k=0$ element of one of the series can again be given as the zero set of two quasihomogeneous functions of weights $w_1,w_2,w_3,w_4$ and degrees $d_1,d_2$. We are now ready to state the following analogue of \cite[Theorem~6]{ET3}:
\begin{theorem} 
Let $X$ be a virtual ICIS and $\widetilde{X}_0$ be the $k=0$ element of the dual series.  Then we have
\begin{equation}
\prod_{j=1}^2 \overline{\zeta}_{X,j}(t) = P_{\widetilde{X}_0}(t) \cdot {\rm Or}_{\widetilde{X}_0}(t) = \prod_{j=1}^2 \overline{\zeta}_{\widetilde{X}_0,j}^\ast(t).
\label{eq:SaiICIS}
\end{equation}
\end{theorem}
\begin{proof}
By Equation~(\ref{eq:piS}), the left-hand side of Equation~(\ref{eq:SaiICIS}) is equal to $\Delta(S_{\gamma_1,\gamma_2,\gamma_3,\gamma_4})(t)$.
By \cite[p.~98]{E1}, there is the following formula for this polynomial :
\begin{equation}
\Delta(S_{\gamma_1,\gamma_2,\gamma_3,\gamma_4})(t) = (t^3-2t^2-2t+1) \prod_{i=1}^4 \frac{t^{\gamma_i}-1}{t-1} + t^2 \sum_{i=1}^4 \frac{t^{\gamma_i-1}-1}{t-1} \prod_{j=1, \atop j \neq i}^4  \frac{t^{\gamma_j}-1}{t-1}.
\label{eq:charPi}
\end{equation}
Using this formula, we can compute the polynomial $\prod_{j=1}^2 \overline{\zeta}_{X,j}(t)$ in each case. The result is given in Table~\ref{TabQmon}. Under a certain non-degeneracy condition, the function $\zeta_{X,2}(t)$ can also be computed by the formula of \cite[Theorem~4]{Gu} from the Newton polytope.

On the other hand, we can compute the Poincar\'e series from the weights and degrees of the dual ICIS given in Table~\ref{TabDual}. The polynomial ${\rm Or}_{\widetilde{X}}(t)$ is given by
$${\rm Or}_{\widetilde{X}}(t) = (1-t)^{-2} \prod_{i=1}^4 (1-t^{\gamma_i})$$
where $\gamma_1,\gamma_2; \gamma_3, \gamma_4$ are the Gabrielov numbers of $X$ which are the Dolgachev numbers of $\widetilde{X}$. Comparing these polynomials, we obtain the first equality of Equation~(\ref{eq:SaiICIS}).
\begin{table}[h]
\begin{center}
\begin{tabular}{|c|c|c|c|c|c|}
\hline
$X$ & ${\rm Dol}(X)$ &  ${\rm Gab}(X)$  & $\prod_{j=1}^2 \overline{\zeta}_{X,j}(t)$ & Weight system $\widetilde{X}_0$ & Dual  \\
\hline
$J'_{2,-1}$ & $2,2;2,6$ & $2,2;2,6$ & $2^2 \cdot 8 \cdot 10/1^2 \cdot 4 \cdot 5$  & $2,6,5,4;8,10$ & $J'_{2,-1}$  \\
$K'_{1,-1}$ & $2,2;4,4$  & $2,2;4,4$ &  $2 \cdot 6 \cdot 8/1^2 \cdot 3$ & $2,4,4,3;6,8$ & $K'_{1,-1}$ \\
$K^\flat_{1,-1}$ & $2,4;2,4$ & $2,2;3,5$  & $2 \cdot 7 \cdot 8/1^2 \cdot 4$ & $2,5,4,3;7,8$ & $L_{1,-1}$ \\
$L_{1,-1}$ & $2,2;3,5$  & $2,4;2,4$ & $2 \cdot 6 \cdot 8/1^2 \cdot 3$ & $2,4,4,3;6,8$ & $K^\flat_{1,-1}$  \\
$L^\sharp_{1,-1}$ & $2,3;2,5$  & $2,3;2,5$ & $2 \cdot 7 \cdot 8/1^2 \cdot 4$  & $2,5,4,3;7,8$ & $L^\sharp_{1,-1}$ \\
$M_{1,-1}$ & $3,3;2,4$  & $3,3;2,4$ & $6 \cdot 7/1^2$  & $3,4,2,3;6,7$ & $M_{1,-1}$ \\
$M^\sharp_{1,-1}$ & $2,3;3,4$   & $2,3;3,4$ & $6 \cdot 7/1^2$  & $2,4,3,3;6,7$ & $M^\sharp_{1,-1}$ \\
$I_{1,-1}$ & $3,3;3,3$ & $3,3;3,3$  & $3 \cdot 6^2/1^2 \cdot 2$ &  $3,3,3,2;6,6$ & $I_{1,-1}$ \\
\hline
\end{tabular}
\end{center}
\caption{Dolgachev numbers, Gabrielov numbers and monodromy zeta functions} \label{TabQmon}
\end{table}

The second equality follows from \cite{EG}.
\end{proof}

In each case, the polynomial $\prod_{j=1}^2 \overline{\zeta}_{\widetilde{X}_0,j}(t)$ has already been indicated in \cite[Table~7]{E3a} under the heading $\pi^\ast$. 

\begin{remark}  The spectrum of an ICIS was defined in \cite{ESt}. In a similar way one can define the spectrum of a virtual ICIS $X$. Spectra for the series of ICIS above have been calculated by Steenbrink \cite{St2}. The spectrum of a virtual ICIS agrees with the spectrum defined by setting $k=-1$ in the corresponding formulas of \cite{St2}. 
The spectral numbers coincide with the exponents of the roots of $\prod_{j=1}^2 \overline{\zeta}_{X,j}(t)$.
\end{remark} 
\begin{sloppypar}

{\bf Acknowledgements}.\  
This work has been partially supported by DFG.
The second named author is also supported 
by JSPS KAKENHI Grant Number JP16H06337. 
\end{sloppypar}


\end{document}